\documentclass[reqno, 11pt]{amsart}
\usepackage{amsmath}
 \usepackage{amssymb}
\usepackage{amsthm}
\usepackage{times}
\usepackage{latexsym}
\usepackage[mathscr]{eucal}
\usepackage{amscd}
\numberwithin{equation}{section}
 
  \newtheorem{theorem}{Theorem}[section]
  \newtheorem{proposition}[theorem]{Proposition}
  \newtheorem{lemma}[theorem]{Lemma}
  \newtheorem{corollary}[theorem]{Corollary}

  \newtheorem{remark}[theorem]{Remark}
  
  \newtheorem{definition}[theorem]{Definition}
  \newtheorem{example}[theorem]{Example}

\title[On three dimensional affine Szab\'o manifolds]{On three dimensional affine Szab\'o manifolds}

\author[Abdoul Salam DIALLO, Silas Longwap, Fortun\'{e} Massamba ]{Abdoul Salam Diallo*, Silas Longwap**, Fortun\'{e} Massamba***}

\newcommand{\acr}{\newline\indent}

\address{\llap{*\,} School of Mathematics, Statistics and Computer Science\acr
 University of KwaZulu-Natal\acr
 Private Bag X01, Scottsville 3209\acr
South Africa  \acr
and \acr
Universit\'e Alioune Diop de Bambey\acr
UFR SATIC, D\'epartement de Math\'ematiques\acr
B. P. 30, Bambey, S\'en\'egal}
\email{Diallo@ukzn.ac.za, abdoulsalam.diallo@uadb.edu.sn}
\address{\llap{**\,}  School of Mathematics, Statistics and Computer Science\acr
 University of KwaZulu-Natal\acr
 Private Bag X01, Scottsville 3209\acr
South Africa} \email{longwap4all@yahoo.com}
\address{\llap{***\,} School of Mathematics, Statistics and Computer Science\acr
 University of KwaZulu-Natal\acr
 Private Bag X01, Scottsville 3209\acr
South Africa} \email{massfort@yahoo.fr, Massamba@ukzn.ac.za}
\thanks{}



\begin{document}

\begin{abstract}
In this paper, we consider the cyclic parallel Ricci tensor condition, which is a necessary condition
for an affine manifold to be Szab\'o. We show that, in dimension $3$, there are affine manifolds which
satisfy the cyclic parallel Ricci tensor but are not Szab\'o. Conversely, it is known that in dimension
$2$, the cyclic parallel Ricci tensor forces the affine manifold to be Szab\'o. Examples of $3$-dimensional 
affine Szabo manifolds are also given. Finally, we give some properties of Riemannian extensions
defined on the cotangent bundle over an affine Szab\'o manifold.
\end{abstract}

\maketitle

\section{Introduction}
The theory of connection is a classical topic in differential geometry. It was initially developed to 
solve pure geometrical problems. It provides an extremely important tool to study geometrical structures on 
manifolds and, as such, has been applied with great sources in many different settings. B. Opozda in 
(\cite{op}) classified locally homogeneous connection on $2$-dimensional manifolds equipped with torsion 
free affine connection. T. Arias-Marco and O. Kowalski  \cite{ak} classify locally homogeneous connections 
with arbitrary torsion on $2$-dimensional manifolds. E. Garc\'ia-Rio \textit{et al.}  \cite{gar} introduced the notion 
of the affine Osserman connections. The affine Osserman connections are well understood in dimension 
two (see  \cite{di,gar} for more details).

A (Pseudo) Riemannian manifold $(M,g)$ is said to be Szab\'o if the eigenvalues of the Szab\'o operator given by 
$$S(X):Y\rightarrow (\nabla_{X}R)(Y,X)X$$ are constants on the unit (Pseudo) sphere bundle, where $R$
denoting the curvature tensor (see \cite{broz} and \cite{gi} for details). The Szab\'o operator is a self adjoint operator with $S(X)X=0$. It plays an important role in the study of totally isotropic manifolds \cite{gis}. Szab\'o in  \cite{sz1} used techniques from algebraic topology to show, in the Riemannian setting, that any such a metric is locally symmetric. He used this observation to  
prove that any two point homogeneous space is either flat or is a rank one symmetric space. Subsequently 
Gilkey and Stravrov  \cite{gs} extended this result to show that any Szab\'o Lorentzian manifold has constant 
sectional curvature. However, for metrics of higher signature the situation is different. Indeed it was 
showed in  \cite{gis} the existence of Szab\'o Pseudo-Riemannian manifolds endowed with metrics of signature $(p,q)$ 
with $p\geq 2$ and $q\geq 2$ which are not locally symmetric .

In \cite{dm}, the authors introduced the so-called \textit{affine Szab\'o connections}. They proved, in 
dimension $2$, that an affine connection $\nabla$ is affine Szab\'o if and only if the Ricci tensor of $\nabla$ is 
cyclic parallel while in dimension $3$ the concept seems to be very challenging by giving only partial results. 

The aim of the present paper is to give an explicit form of two families of affine connections which are affine Szab\'o on $3$-dimensional manifolds. Moreover, although both results provide examples of affine Szab\'o connections, they are essentially different in nature since, in the first family, the affine 
Szab\'o condition coincides with the cyclic parallelism of the Ricci tensor, whereas the second one is not. For any affine connection $\nabla$ on $M$,
there exist a technique called \textit{Riemannian extension}, which relates affine and pseudo-Riemannian geometries.
This technique is very powerful in constructing new examples of pseudo-Riemannian metrics. The relation between affine Szab\'o
manifolds and pseudo-Riemannian Szab\'o manifolds are investigated by using  Riemannian extensions.

The paper is organized as follows. In section \ref{Prem}, we recall some basic definitions and geometric objects, 
namely, torsion, curvature tensor, Ricci tensor and affine Szab\'o operator on an affine manifold. In section \ref{RiccCycl}, we study the cyclic parallelism of the Ricci tensor for two particular cases of affine connections in $3$-dimensional affine manifolds. We establish geometrical configurations of affine manifolds admitting a cyclic parallel Ricci tensor (Propositions 3.2 and 3.4). In section \ref{Szabo}, we study the Szab\'o 
condition on two particular affine connections (Theorems 4.5 and 4.7). Finally, we end the paper in section 5 by investigating the Riemannian extensions defined on the cotangent bundle over an affine Szab\'o manifold.

\section{Preliminaries}\label{Prem}

Let $M$ be an $n$-dimensional smooth manifold and $\nabla$ be an affine connection on $M$. We consider a 
system of coordinates $(x_{1},x_{2},\cdots,x_{n})$ in a neighborhood $\mathcal{U}$ of a point $p$ in $M$. In 
$\mathcal{U}$ the affine connection is given by
\begin{equation}
\nabla_{\partial_{i}}\partial_{j} =f^{k}_{ij}\partial_{k}
\end{equation}
where $\{\partial_{i}=\frac{\partial}{\partial x_{i}}\}_{1\leq i\leq n}$ is a basis of the tangent space 
$T_{p} M$ and the functions $f_{ij}^{k}\, (i,j,k=1,2,3,\cdots,n)$ are called the \textit{Christoffel symbols} of 
the affine connection. We shall call the pair $(M,\nabla)$ \textit{affine manifold}. Some tensor fields associated 
with the given affine connection $\nabla$ are defined below.

The \textit{torsion tensor} field $T^{\nabla}$ is defined by
\begin{equation}
T^{\nabla}(X,Y)=\nabla_{X}Y-\nabla_{Y}X-\nabla_{[X,Y]}
\end{equation}
for any vector fields $X$ and $Y$ on $M$. The components of the torsion tensor $T^{\nabla}$ in local coordinates are
\begin{equation}
T_{ij}^{k}=f_{ij}^{k}-f^{k}_{ji}.
\end{equation}
If the torsion tensor of a given affine connection $\nabla$ vanishes, we say that $\nabla$ is torsion-free

The \textit{curvature tensor field} $\mathcal{R}^{\nabla}$ is defined by 
\begin{equation}
\mathcal{R}^{\nabla}(X,Y)=\nabla_{X}\nabla_{Y}Z-\nabla_{Y}\nabla_{X}Z-\nabla_{[X,Y]}Z
\end{equation}
for any vector field $X, Y$ and $Z$ on $M$. The components in local coordinates are
\begin{equation}
\mathcal{R}^{\nabla}(\partial_{k},\partial_{l})=\sum_{i}R^{i}_{jkl}\partial_{i}.
\end{equation}
We shall assume that $\nabla$ is torsion-free. If $\mathcal{R}^{\nabla}=0$ on $M$, we say that $\nabla$ is flat 
affine connection. It is known that $\nabla$ is flat if and only if around a point $p$ there exist a local coordinate 
system such that $f_{ij}^{k}=0$ for all $i,j,k$.

We define \textit{Ricci tensor} $Ric^{\nabla}$ by
\begin{equation}
Ric^{\nabla}(X,Y)= \mathrm{trace}\{Z\mapsto \mathcal{R}^{\nabla}(Z,X)Y\}.
\end{equation} 
The components in local coordinates are given by
\begin{eqnarray}\label{CurRicc1}
Ric^{\nabla}(\partial_{j},\partial_{k})=\sum_{i}R^{i}_{kij} .
\end{eqnarray}
It is known that in Riemannian geometry the Levi-Civita connection of a Riemannian metric has symmetric 
Ricci tensor, that is $Ric^{\nabla}(X,Y)=R^{\nabla}(Y,X)$. But this property is not true for an arbitrary 
torsion-free affine connection. In fact, the property is closely related to the concept of parallel volume 
element. (See ~\cite{ns} for more details).

The covariant derivative of the curvature tensor $\mathcal{R}^{\nabla}$ is given by
\begin{align*}
 (\nabla_{X}\mathcal{R}^{\nabla})(Y, Z)W &=  \nabla_{X}\mathcal{R}^{\nabla}(Y,Z)W
 -\mathcal{R}^{\nabla}(\nabla_X Y,Z)W \\
 &-\mathcal{R}^{\nabla}(Y,\nabla_X Z)W   -\mathcal{R}^{\nabla}(Y,Z) \nabla_X W.
\end{align*}
The covariant derivative of the Ricci tensor $Ric^{\nabla}$ is defined by 
\begin{equation}\label{CoDerRicNa1}
(\nabla_{X}Ric^{\nabla})(Z,W)=X(Ric^{\nabla}(Z,W))-Ric^{\nabla}(\nabla_{X}Z,W)-Ric^{\nabla}(Z,\nabla_{X}W).
\end{equation}
For $X\in\Gamma(T_{p}M)$, we define the \textit{affine Szab\'o operator} $S^{\nabla}(X):T_p M\rightarrow T_p M$ with 
respect to $X$ by 
\begin{equation}
S^{\nabla} (X)Y :=(\nabla_{X} \mathcal{R}^{\nabla})(Y,X)X 
\end{equation}
for any vector field $Y$. The affine Szab\'o operator satisfies $S^{\nabla}(X)X=0$ and 
$S^{\nabla}(\beta X)=\beta^{3}S^{\nabla}(X)$ for $\beta\in \mathbb{R}-\{0\}$ and $X\in T_{p}M$. If $Y=\partial_{m}$, for
$m=1,2,\cdots,n$ and $X= \sum_{i}\alpha_{i}\partial_{i}$ one get
\begin{equation}
S^{\nabla}(X)\partial_{m}=\sum_{i,j,k=1}^{n}\alpha_{i}\alpha_{j}\alpha_{k}(\nabla_{\partial_{i}}\mathcal{R}^{\nabla})(\partial_{m},\partial_{j})\partial_{k}.
\end{equation}
Note that, by definition of the Ricci tensor, one has
\begin{equation}
 \mathrm{trace}(Y\mapsto (\nabla_{X} \mathcal{R}^{\nabla})(Y,X)X)=(\nabla_{X}Ric^{\nabla})(X,X).
\end{equation}

\section{Affine connections with cyclic parallel Ricci tensor}\label{RiccCycl}

In this section, we investigate affine connections whose Ricci tensors are cyclic parallel. We shall consider two
cases of $3$-dimensional smooth manifolds with specific affine connections. We start with a formal definition.

\begin{definition}{\rm The Ricci tensor $Ric^{\nabla}$ of an affine manifold $(M,\nabla)$ is cyclic parallel if
\begin{equation}\label{RicciCycP1}
(\nabla_{X}Ric^{\nabla})(X,X)=0,
\end{equation}
for any vector field $X$ tangent to $M$ or, equivalently, if
$$
 \mathfrak{G}_{X,Y,Z}(\nabla_{X}Ric^{\nabla})(Y,Z)=0,
$$
for any vector fields $X,Y,$ and $Z$ tangent to $M$ where $\mathfrak{G}_{X,Y,Z}$ denotes the cyclic sum with respect 
to $X,Y$ and $Z$.
}
\end{definition}
Locally, the equation (\ref{RicciCycP1}) takes the form
\begin{equation}
(\nabla_{\partial_{i}}Ric^{\nabla})_{jk}=0
\end{equation}
or can be written out without the symmetrizing brackets
\begin{equation}
(\nabla_{\partial_{i}}Ric^{\nabla})_{jk}+(\nabla_{\partial_{j}}Ric^{\nabla})_{ki}+(\nabla_{\partial_{k}}Ric^{\nabla})_{ij}=0.
\end{equation}
For $X=\sum_{i}\alpha_{i}\partial_{i}$ , it is easy to show that
\begin{equation}\label{RicciCycP2}
(\nabla_{X}Ric^{\nabla})(X,X)=\sum_{i,j,k}\alpha_{i}\alpha_{j}\alpha_{k}(\nabla_{\partial_{i}}Ric^{\nabla})_{jk}.
\end{equation}
Now, we are going to present two cases of affine connections in which we investigate the cyclic parallelism of the Ricci 
tensor.   

\textit{Case 1:}  
Let $M$ be a $3$-dimensional smooth manifold and $\nabla$ be an affine torsion-free connection. Suppose that the action 
of the affine connection $\nabla$ on the basis of the tangent space $\{\partial_{i}\}_{1\le i\le 3}$ is given by 
\begin{equation}\label{CoefCon1}
  \nabla_{\partial_{1}}\partial_{1} =f_{1} \partial_{1},\;\;\nabla_{\partial_{1}}\partial_{2} =f_{2} \partial_{1}\;\;\mbox{and}\;\;\nabla_{\partial_{1}}\partial_{3} = f_{3} \partial_{1},
\end{equation} 
where the smooth functions$f_{i} =f_{i}(x_1,x_2,x_3)$ are Christoffel symbols. The non-zero components of the curvature tensor 
$\mathcal{R}^{\nabla}$ of the affine connection (\ref{CoefCon1}) are given by
\begin{align*}
\mathcal{R}^{\nabla}(\partial_{1},\partial_{2})\partial_{1}&= (\partial_{1}f_{2}-\partial_{2}f_{1})\partial_{1}, \;\;\;
\mathcal{R}^{\nabla}(\partial_{1},\partial_{2})\partial_{2} = -(\partial_{2}f_{2}+f_{2}^{2})\partial_{1}, \\
\mathcal{R}^{\nabla}(\partial_{1},\partial_{2})\partial_{3}&= -(\partial_{2}f_{3}+f_{2}f_{3})\partial_{1},\;\;\;
\mathcal{R}^{\nabla}(\partial_{1},\partial_{3})\partial_{1}= (\partial_{1}f_{3}-\partial_{3}f_{1})\partial_{1},\\
\mathcal{R}^{\nabla}(\partial_{1},\partial_{3})\partial_{2}&= -(\partial_{3}f_{2}+f_{2}f_{3})\partial_{1},\;\;\;
\mathcal{R}^{\nabla}(\partial_{1},\partial_{3})\partial_{3} = -(\partial_{3}f_{3}+f_{3}^{2})\partial_{1}\nonumber\\
\mathcal{R}^{\nabla}(\partial_{2},\partial_{3})\partial_{1}&= (\partial_{2}f_{3}-\partial_{3}f_{2})\partial_{1}.
\end{align*}
From (\ref{CurRicc1}), the non-zero components of the Ricci tensor $Ric^{\nabla}$ of the affine connection (\ref{CoefCon1}) 
are given by:
\begin{align*}
Ric^{\nabla}(\partial_2,\partial_1) &=  \partial_1f_2-\partial_2f_1, \;\;\;
Ric^{\nabla}(\partial_2,\partial_2)  =  -(\partial_2f_2+f_2^2),\\
Ric^{\nabla}(\partial_2,\partial_3) & =  -(\partial_2f_3+f_2f_3),\;\;\;
Ric^{\nabla}(\partial_3,\partial_1)  = \partial_1f_3-\partial_3f_1,\\
Ric^{\nabla}(\partial_3,\partial_2) &=  -(\partial_3f_2+f_2f_3),\;\;\;
Ric^{\nabla}(\partial_3,\partial_3)  = -(\partial_3f_3+f_3^2).
\end{align*}

\begin{proposition}\label{proposition}
On $\mathbb{R}^{3}$, the affine connection $\nabla$ defined in (\ref{CoefCon1}) satisfies the relation (\ref{RicciCycP1}) if the 
functions $f_i=f_{i}(x_1,x_2,x_3)$, for $i=1,2,3$, satisfy the following partial differential equations:
\begin{align}
&\partial^2_3f_3+2f_3\partial_3f_3=0\nonumber\\
&\partial^2_2f_2+2f_2\partial_2f_2=0\nonumber\\
&\partial^2_3f_1+4f_3\partial_1f_3-2f_3\partial_3f_1=0\nonumber\\
&\partial^2_2f_1+4f_2\partial_1f_2-2f_2\partial_2f_1=0\nonumber\\
&\partial^2_1f_3-\partial_1\partial_3f_1-f_1\partial_1f_3+f_1\partial_3f_1=0\nonumber\\
&\partial^2_1f_2-\partial_1\partial_2f_1-f_1\partial_1f_2+f_1\partial_2f_1=0\nonumber\\
&\partial^2_2f_3+2\partial_3\partial_2f_2+2f_2\partial_3f_2+2f_3\partial_2f_2+2f_2\partial_2f_3=0\nonumber\\
&\partial^2_3f_2+2\partial_3\partial_2f_3+2f_3\partial_2f_3+2f_3\partial_3f_2+2f_2\partial_3f_3=0\nonumber\\
&4f_3\partial_1f_2+4f_2\partial_1f_3-2f_3\partial_2f_1-2f_2\partial_3f_1+2\partial_3\partial_2f_1=0
\end{align}
\end{proposition}

\begin{proof}
From a straightforward calculation, using (\ref{CoDerRicNa1}) and (\ref{RicciCycP2}), one obtain the result.
\end{proof}

As an example to the Proposition \ref{proposition}, we have the following.
\begin{example}{\rm 
The Ricci tensors of the affine connections defined in (\ref{CoefCon1}) on $\mathbb{R}^3$ with
\begin{enumerate}
 \item $f_1=0, f_2=-x_3$ and $f_3 = x_2$;
 \item $f_1 =x_1, f_2 =2x_3$ and $f_3=-2x_2$
\end{enumerate}
are cyclic parallel.}
\end{example}

\textit{Case 2:}  
Let $M$ be a $3$-dimensional smooth manifold and $\nabla$ be an affine torsion-free connection. Suppose that 
the action of the affine connection $\nabla$ on the basis of the tangent space $\{\partial_{i}\}_{1\le i\le 3}$ is 
given by 
\begin{equation}\label{CoefCon2}
 \nabla_{\partial_{1}}\partial_{1} = f_{1} \partial_{2},\;\;\nabla_{\partial_{2}}\partial_{2} = f_{2} \partial_{3}\;\;\mbox{and}\;\;\nabla_{\partial_{3}}\partial_{3} = f_{3} \partial_{1},
\end{equation} 
where smooth functions $f_{i} =f_{i}(x_1,x_2,x_3)$ are Christoffel symbols . The non-zero components of the curvature tensor 
$\mathcal{R}^{\nabla}$ of the affine connection (\ref{CoefCon2}) are given by
\begin{align*}
 \mathcal{R}^{\nabla}(\partial_{1},\partial_{2})\partial_{1}& =-(\partial_{2}f_{1}\partial_{2}+f_{1}f_{2}\partial_{3}), \;\;\;
\mathcal{R}^{\nabla}(\partial_{1},\partial_{2})\partial_{2}=\partial_{1}f_{2}\partial_{3},\\
 \mathcal{R}^{\nabla}(\partial_{1},\partial_{3})\partial_{1}&=-\partial_{3}f_{1}\partial_{2},\;\;\;
\mathcal{R}^{\nabla}(\partial_{1},\partial_{3})\partial_{3}=\partial_{1}f_{3}\partial_{1}+f_{1}f_{3}\partial_{2},\\
 \mathcal{R}^{\nabla}(\partial_{2},\partial_{3})\partial_{2}&=-(\partial_{3}f_{2}\partial_{3}+f_{3}f_{2}\partial_{1}),\;\;\;
\mathcal{R}^{\nabla}(\partial_{2},\partial_{3})\partial_{3}=\partial_{2}f_{3}\partial_{1}.
\end{align*}
From (\ref{CurRicc1}), the non-zero components of the Ricci tensor $Ric^{\nabla}$ of the affine connection (\ref{CoefCon2}) 
are given by $Ric^{\nabla}(\partial_{1},\partial_{1})=\partial_{2}f_{1}$, $Ric^{\nabla}(\partial_{2},\partial_{2})=\partial_{3}f_{2}$, $ Ric^{\nabla}(\partial_{3},\partial_{3})=\partial_{1}f_{3}$. 
\begin{proposition}\label{PropoNewfunc}
The affine connection $\nabla$ defined on $\mathbb{R}^{3}$ by (\ref{CoefCon2}) satisfies (\ref{RicciCycP1}) if the 
functions $f_i=f_{i}(x_1,x_2,x_3)$, for $i=1,2,3$, has the following form:
\begin{equation*}
 f_1 =f(x_1) + g(x_3),\;\;
 f_2 = h(x_1) + u(x_2),\;\;
 f_3 = v(x_2) + t(x_3),
\end{equation*}
where $f$, $g$, $h$, $u$, $v$ and $t$ are smooth functions on $\mathbb{R}^{3}$.
\end{proposition}
\begin{proof}
From a straightforward calculation, using (\ref{CoDerRicNa1}) and (\ref{RicciCycP2}), one obtain 
the following partial differential equations:
\begin{align}
&\partial_{1}\partial_{2}f_{1}=0, \;\;\; \partial_{3}\partial_{2}f_{1}=0, \;\; \partial_{1}\partial_{3}f_{2}=0, \;\; \partial_{2}\partial_{3}f_{2}=0, \;\; \partial_{2}\partial_{1}f_{3}=0,\nonumber\\
&\partial_{3}\partial_{1}f_{3}=0, \;\; \partial^{2}_{2}f_{1}-2f_{1}\partial_{3}f_{2}=0,\;\;\partial^{2}_{1}f_{3}-2f_{3}\partial_{2}f_{1}=0, \;\;\partial^{2}_{3}f_{2}-2f_{2}\partial_{1}f_{3}=0,\nonumber
\end{align}
and the result follows.
\end{proof}
As an application to this proposition, we have:
\begin{example}\label{ExamplImport}{\rm 
The Ricci tensor of the following affine connection defined on $\mathbb{R}^3$ by
\begin{eqnarray*}
 \nabla_{\partial_1} \partial_1 = x^{2}_{1} \partial_2,\quad
 \nabla_{\partial_2} \partial_2 = (x_1 + x_2) \partial_3,\quad \mbox{and}\quad
 \nabla_{\partial_3} \partial_3 &=& (x_2 + x^{2}_{3}) \partial_1,
\end{eqnarray*}
is cyclic parallel.}
\end{example}

The manifolds with cyclic parallel Ricci tensor, known as $L_{3}$-spaces, are well-developed in Riemannian 
geometry. The cyclic parallelism of the Ricci tensor is sometime called the ``\textit{First Ledger condition}'' 
\cite{pt}. In \cite{sz2}, for instance, the author proved that a smooth Riemannian manifold satisfying the 
first Ledger condition is real analytic. These Riemannian manifolds were introduced by A. Gray in (~\cite{gr}) 
as a special subclass of (connected) Riemannian manifolds $(M,g)$, called Einstein-like spaces, all of which 
have constant scalar curvature. Also, Riemannian manifolds of dimension $3$ with cyclic parallel Ricci tensor 
are locally homogeneous naturally reductive (~\cite{pt}). Tod in \cite{to} used the same condition to characterize 
the $4$-dimensional K\"ahler manifolds which are not Einstein. It has also enriched the D'Atri spaces 
(see \cite{ks,pt} for more details).

\section{The affine Szab\'o manifolds}\label{Szabo}

Let $(M,\nabla)$ be an $n$-dimensional affine manifold, i.e., $\nabla$ is a torsion free
connection on the tangent bundle of a smooth manifold $M$ of dimension $n$. Let $\mathcal{R}^{\nabla}$
be the associated curvature operator. We define the \textit{affine Szab\'o operator}
$S^{\nabla}(X):T_p M\rightarrow T_p M$ with respect to a vector $X\in T_p M$ by
\begin{equation*}
 S^{\nabla}(X) Y := (\nabla_X \mathcal{R}^{\nabla})(Y,X)X.
\end{equation*}

\begin{definition}{\rm
Let $(M,\nabla)$ be a smooth affine manifold and $p\in M$.
\begin{enumerate}
\item  $(M,\nabla)$ is called affine Szab\'o at $p\in M$ if the affine Szab\'o  operator $S^{\nabla}(X)$ has the 
same characteristic polynomial for every vector field $X$ on $M$.
\item Also, $(M,\nabla)$ is called affine Szab\'o if $(M,\nabla)$ is affine Szab\'o at each point $p\in M$.
\end{enumerate}
}
\end{definition}
\begin{theorem}{\rm
Let $(M,\nabla)$ be an $n$-dimensional affine manifold and $p\in M$. Then $(M,\nabla)$ is affine Szab\'o at 
$p\in M$ if and only if the characteristic polynomial of the affine Szab\'o operator $S^{\nabla}(X)$ is 
$$P_{\lambda}(S^{\nabla}(X))=\lambda^{n},$$ 
for every $X\in T_{p}M$.
}
\end{theorem}

\begin{corollary}{\rm
We say that $(M,\nabla)$ is affine Szab\'o if the affine Szab\'o operators are nilpotent, i.e., $0$ is the 
eigenvalue of $S^{\nabla}(X)$ on the tangent bundle.  
}
\end{corollary}

\begin{corollary}{\rm
 If $(M,\nabla)$ is affine Szab\'o at $p\in M$, then the Ricci
 tensor is cyclic parallel.
 }
\end{corollary}

Affine Szab\'o connections are well-understood in $2$-dimension, due to the fact that an affine connection
is Szab\'o if and only if its Ricci tensor is cyclic parallel \cite{dm}. The situation is however more 
involved in higher dimensions where the cyclic parallelism is a necessary but not sufficient condition for 
an affine connection to be Szab\'o.

Let $X=\sum_{i=1}^{3} \alpha_i \partial_i$ be a vector field on a $3$-dimensional affine manifold $M$. Then the
affine Szab\'o operator is given by
\begin{eqnarray*}
 S^{\nabla}(X) (\partial_{m}) = \sum_{i,j,k=1}^{3}\alpha_{i}\alpha_{j}\alpha_{k}(\nabla_{\partial_{i}}
 \mathcal{R}^{\nabla})(\partial_{m},\partial_{j})\partial_{k}, \quad m=1,2,3.
\end{eqnarray*}

\subsection{First Family of affine Szab\'o connection.}
Next, we give an example of a family of affine Szab\'o connection on a $3$-dimensional manifold. Let us consider 
the affine connection defined in $(3.5)$, i.e.,
\begin{equation*}
\nabla_{\partial_{1}}\partial_{1}=f_{1}\partial_{1},\quad
\nabla_{\partial_{1}}\partial_{2}=f_{2}\partial_{1}\quad \mbox{and}\quad
\nabla_{\partial_{1}}\partial_{3}=f_{3}\partial_{1},
\end{equation*}
where the smooth functions $f_i = f_i(x_1,x_2,x_3)$ ($i=1,2,3$) are Christoffel symbols. For $ X=\sum_{i=1}^{3}\alpha_{i}\partial_{i}$, the affine 
Szab\'o operator is given by
\begin{align*}
S^{\nabla}(X) (\partial_{1})=a_{11}\partial_1, \quad
S^{\nabla}(X) (\partial_{2})=a_{12}\partial_1 \quad \mbox{and}\quad
S^{\nabla}(X) (\partial_{3})=a_{13}\partial_1
\end{align*}
with
\begin{align}
  a_{11} &=\alpha_{3}^{3}\{\partial^2_3f_3 + 2f_3\partial_3f_3\}
+ \alpha_{2}^{3}\{\partial^2_2f_2 + 2f_2\partial_2f_2\}\nonumber\\
&
+ \alpha_{3}^{2}\alpha_{1}\{\partial^2_3f_1 + 4f_3\partial_1f_3 - 2f_3\partial_3f_1\}\nonumber\\
& + \alpha_{2}^{2}\alpha_{1}\{\partial^2_2f_1 + 4f_2\partial_1f_2 - 2f_2\partial_2f_1\}\nonumber\\
&
+ \alpha_{1}^{2}\alpha_{3}\{\partial^2_1f_3 - \partial_1\partial_3f_1 - f_1\partial_1f_3 
+ f_1\partial_3f_1\}\nonumber\\
& + \alpha_{1}^{2}\alpha_{2}\{\partial^2_1f_2 - \partial_1\partial_2f_1 - f_1\partial_1f_2 
+ f_1\partial_2f_1\}\nonumber\\
& + \alpha_{2}^{2}\alpha_{3}\{\partial^2_2f_3 + 2\partial_3\partial_2f_2 
+ 2f_2\partial_3f_2 + 2f_3\partial_2f_2 + 2f_2\partial_2f_3\}\nonumber\\
& +\alpha_{3}^{2}\alpha_{2} \{\partial^2_3f_2 + 2\partial_3\partial_2f_3 + 2f_3\partial_2f_3
+ 2f_3\partial_3f_2 + 2f_2\partial_3f_3\}\nonumber\\
& +\alpha_{1}\alpha_{2}\alpha_{3}\{4f_3\partial_1f_2 + 4f_2\partial_1f_3 - 2f_3\partial_2f_1 
- 2f_2\partial_3f_1 + 2\partial_3\partial_2f_1\},\nonumber\\
 a_{12}&= \alpha_{1}^{3}\{\partial_{1}\partial_2f_1 - \partial_{1}^{2}f_{2} - f_{1}\partial_{2}f_{1}\nonumber\\
&
+ 2f_{1}\partial_{1}f_{2}\}
+ \alpha_{2}^{2}\alpha_{1}\{\partial^2_2f_2 + 2f_2\partial_2f_2 - f^{3}_{2}\}\nonumber\\
&+\alpha_{3}^{2}\alpha_{1} \{-\partial^2_3f_2f_{3} - 2f_3\partial_2f_3 + 2f_3\partial_3f_2 + 2f_2\partial_3f_3
+ 2\partial_3\partial_2f_3 \}\nonumber\\
& +\alpha_{1}^{2}\alpha_{3}\{2\partial_1\partial_2f_3 - 2\partial_1\partial_3f_2 + \partial_3\partial_2f_1 
+ f_{2}\partial_3f_1 - 3f_3\partial_2f_1 + 4f_3\partial_1f_2\}\nonumber\\
& +\alpha_{1}^{2}\alpha_{2}\{\partial^2_2f_1 - 2f_{2}\partial_2f_1 + 4f_2\partial_1f_2\}
 +\alpha_{1}\alpha_{2}\alpha_{3}\{4f_2\partial_3f_2 + 2\partial_{2}^{2}f_3\},\nonumber
\end{align}
\begin{align}
 a_{13}& = \alpha_{1}^{3}\{\partial_{1}\partial_3f_1 - \partial_{1}^{2}f_{3} - f_{1}\partial_{3}f_{1} 
+ f_{1}\partial_{1}f_{3}\}\nonumber\\
&+ \alpha_{1}^{2}\alpha_{3}\{\partial^2_3f_1 + 3f_3\partial_1f_3 - f_3\partial_3f_1\}\nonumber\\
& + \alpha_{1}^{2}\alpha_{2}\{\partial_{2}\partial_3f_1 + 4f_2\partial_1f_3 - 2\partial_{2}\partial_1f_3 
- 3f_{2}\partial_3f_1 + f_{3}\partial_2f_1\}\nonumber\\
& + \alpha_{2}^{2}\alpha_{1}\{2f_{3}\partial_2f_2 + 2f_2\partial_2f_3 + \partial_{2}\partial_{3}f_{2}\}\nonumber\\
&
+ \alpha_{3}^{2}\alpha_{1}\{\partial_3^{2}f_3 - f_{3}\partial_3f_3 - f^{3}_{3} + 2f_{3}\partial_3f_3\}\nonumber\\
& + \alpha_{1}^{2}\alpha_{3}\{4f_3\partial_2f_3 - 2f_{2}\partial_3f_3 - f_3\partial_3f_2 + 2\partial^{2}_{3}f_{2}
- f^{2}_{3}f_{2}\}.\nonumber
\end{align}
Since the Ricci tensor of any affine Szab\'o connection is cyclic parallel, it follows  that $a_{11}=0$. Thus
the characteristic polynomial of the matrix associated to $S^{\nabla}(X)$ with respect to the basis
$\{\partial_{1},\partial_{2},\partial_{3}\}$ is equal to: 
$$
P_{\lambda}(S^{\nabla}(X))=-\lambda^{3}.
$$
 
We have the following result. 
\begin{theorem} \label{THeorDim3}
Let $M=\mathbb{R}^{3}$  and $\nabla$ be the torsion free affine connection, whose nonzero coefficients of the 
connection are given by
\begin{equation*}
 \nabla_{\partial_{1}}\partial_{1} = f_{1} \partial_{1},\;\;\nabla_{\partial_{1}}\partial_{2} = f_{2} \partial_{1}\;\;\mbox{and}\;\;\nabla_{\partial_{1}}\partial_{3} = f_{3} \partial_{1}.
\end{equation*} 
Then $(M,\nabla)$ is affine Szab\'o if and only if the Ricci tensor of $(M,\nabla)$ is cyclic parallel. 
\end{theorem}
From Theorem \ref{THeorDim3}, one can construct examples of affine Szab\'o connections. 
\begin{example}{\rm
The following affine connections on $\mathbb{R}^3$ whose non-zero Christoffel symbols are given by:
\begin{enumerate}
 \item $\nabla_{\partial_{1}}\partial_{1}=0, \quad \nabla_{\partial_{1}}\partial_{2}=-x_3\partial_1,
 \quad \nabla_{\partial_{1}}\partial_{3}=x_2\partial_1$;
 \item $\nabla_{\partial_{1}}\partial_{1}=x_1\partial_1, \quad \nabla_{\partial_{1}}\partial_{2}=2x_3\partial_1,
 \quad \nabla_{\partial_{1}}\partial_{3}=-2x_2\partial_1$;
\end{enumerate}
are affine Szab\'o.}
\end{example}
Note that the result in Theorem \ref{THeorDim3} remains the same if the affine connection $\nabla$ has non-zero components 
$\nabla_{\partial_{1}}\partial_{1}$, $\nabla_{\partial_{1}}\partial_{2}$ and $\nabla_{\partial_{1}}\partial_{3}$ in the same 
direction of the element of the basis $\{\partial_{i}\}_{i=1.2.3}$.

The affine manifolds in Theorem \ref{THeorDim3} are also called $L_3$-spaces, and Therefore, are d'Atri spaces. We refer to
\cite{ks} for a further discussion of D'Atri spaces.

\subsection{Second Family of affine Szab\'o connection}

Let us consider the affine connection defined in (\ref{CoefCon2}), i.e.,
\begin{equation*}
\nabla_{\partial_{1}}\partial_{1}=f_{1}\partial_{2},\quad
\nabla_{\partial_{2}}\partial_{2}=f_{2}\partial_{3}\quad \mbox{and} \quad
\nabla_{\partial_{3}}\partial_{3}=f_{3}\partial_{1}
\end{equation*}
where the smooth functions $f_i=f_i(x_1,x_2,x_3)$, for $i=1,2,3$,  are Christoffel symbols. Since the Ricci tensor of any affine Szab\'o connection is
cyclic parallel, it follows from the Proposition \ref{PropoNewfunc}, that the matrix associated to the affine Szab\'o operator 
with respect to the basis $\{\partial_{1},\partial_{2},\partial_{3}\}$ is reduced to
\begin{eqnarray*}
(S^{\nabla})(X)=
\left(\begin{array}{ccc}
0&b_{12}&b_{13}\\                                              
b_{21}&0&b_{23}\\                                              
b_{31}&b_{32}&0                                              
\end{array}
\right)
\end{eqnarray*}
with
\begin{align*}
b_{12}&= \alpha^{2}_{1}\alpha_3(-\partial_{1}\partial_{3}f_{1})
+\alpha_1\alpha^{2}_{2}(f_{2}\partial_{3}f_{1})
+ \alpha_1\alpha^{2}_{3}(f_{3}\partial_{1}f_{1} - \partial_{3}^{2}f_{1})\\
&+ \alpha^{2}_{2}\alpha_3(-2f_{2}f_{3}f_{1})
+\alpha_2\alpha^{2}_{3}(f_{1}\partial_{2}f_{3})
+\alpha^{3}_{3}(2f_{3}\partial_{3}f_{1}+f_{1}\partial_{3}f_{3});\\
b_{13}&= \alpha^{2}_{1}\alpha_2 (-2f_{1}\partial_{1}f_{2}-f_{2}\partial_{1}f_{1})
+\alpha_1\alpha^{2}_{2}(\partial_{1}^{2}f_{2}-f_{1}\partial_{2}f_{2})\\
&+\alpha_1\alpha_2\alpha_3(-2f_2\partial_{3}f_{1}) + \alpha^{3}_{2}(\partial_{2}\partial_{1}f_{2})
+\alpha_2\alpha^{2}_{3}(2f_{2}f_{3}f_{1});\\
b_{21}&= \alpha^{2}_{1}\alpha_3(2f_{3}f_{2}f_{1})
+\alpha_1\alpha_2\alpha_3(-2f_{3}\partial_{1}f_{2})
+\alpha^{2}_{2}\alpha_3(-2f_{2}\partial_{2}f_{3}-f_{3}\partial_{2}f_{2})\\
&+ \alpha_2\alpha^{2}_{3}(\partial_{2}^{2}f_{3}-f_2\partial_{3}f_{3})
+\alpha^{3}_{3}(\partial_{3}\partial_{2}f_{3});
\end{align*}
\begin{align*}
b_{23}&= \alpha^{3}_{1}(2f_{1}\partial_{1}f_{2}+f_{2}\partial_{1}f_{1})
+\alpha^{2}_{1}\alpha_2(-\partial_{1}^{2}f_{2}+f_{1}\partial_{2}f_{2})
+\alpha^{2}_{1}\alpha_3(f_{2}\partial_{3}f_{1})\\
&+ \alpha_1\alpha^{2}_{2}(-\partial_{2}\partial_{1}f_{2})
+\alpha_1\alpha^{2}_{3}(-2f_{2}f_{3}f_{1})
+\alpha_2\alpha^{2}_{3}(f_{3}\partial_{1}f_{2});\\
b_{31}&= \alpha^{2}_{1}\alpha_2(-2f_{1}f_{2}f_{3})
+\alpha^{2}_{1}\alpha_3(f_{1}\partial_{2}f_{3})
+\alpha_1\alpha^{2}_{2}(f_{3}\partial_{1}f_{2})\\
&+ \alpha^{3}_{2}(2f_{2}\partial_{2}f_{3}+f_{3}\partial_{2}f_{2})
+\alpha^{2}_{2}\alpha_3(-\partial_{2}^{2}f_{3}+f_{2}\partial_{3}f_{3})
+\alpha_2\alpha^{2}_{3}(-\partial_{3}\partial_{2}f_{3});\\
b_{32}&= \alpha^{3}_{1}(\partial_{1}\partial_{3}f_{1})
+\alpha^{2}_{1}\alpha_3(-f_{3}\partial_{1}f_{1}+\partial_{3}^{2}f_{1})
+\alpha_1\alpha^{2}_{2}(2f_{1}f_{3}f_{2})\\
&+ \alpha_1\alpha_2\alpha_3(-2f_1\partial_{2}f_{3})
+\alpha_1\alpha^{2}_{3}(-2f_{3}\partial_{3}f_{1}-f_{1}\partial_{3}f_{3}).
\end{align*}
The characteristic polynomial of the affine Szab\'o operator is now seen to be:
\begin{eqnarray*}
P[S^{\nabla}(X)] (\lambda) =- \lambda^{3}
 +(b_{12}b_{21} +b_{23}b_{32} + b_{13}b_{31})\lambda
 +(b_{12}b_{23}b_{31} + b_{13}b_{21}b_{32}).
\end{eqnarray*}
It follows that the affine connection given by (\ref{CoefCon2}) is affine Szab\'o if and only if:
\begin{eqnarray*}
 b_{12}b_{21} +b_{23}b_{32} + b_{13}b_{31} =0 \quad \mbox{and}\quad b_{12}b_{23}b_{31} + b_{13}b_{21}b_{32}=0.
\end{eqnarray*}
A straightforward calculation shows that: $b_{12}b_{23}b_{31} + b_{13}b_{21}b_{32}=0$. Then
$S^{\nabla}(X)$ has eigenvalue zero if and only if:
\begin{equation}\label{EquationCond}
 b_{12}b_{21} +b_{23}b_{32} + b_{13}b_{31} =0.
\end{equation}
\begin{enumerate}
\item Assume $f_1 =0$. Then, the relation (\ref{EquationCond}) reduces to:
\begin{eqnarray*}
  b_{13}b_{31} =0.
\end{eqnarray*}
\begin{enumerate}
 \item If $\partial_1 f_2 =0$, then $f_2 = u(x_2)$ and $f_3= v(x_2)+ t(x_3)$.
 \item If $\partial_1 f_2 \neq 0$, then $f_3 = 0$.
\end{enumerate}
\item Assume $f_2=0$, then we have
\begin{eqnarray*}
  b_{12}b_{21} =0.
\end{eqnarray*}
\begin{enumerate}
 \item If $\partial_2 f_3 =0$, then $f_3 = t(x_3)$ and $f_1= f(x_1)+g(x_3)$.
 \item If $\partial_2 f_3 \neq 0$, then $f_1 = 0$.
\end{enumerate}
\item Assume $f_3=0$, then we have
\begin{eqnarray*}
  b_{23}b_{32} =0.
\end{eqnarray*}
\begin{enumerate}
 \item If $\partial_3 f_1 =0$, then $f_1 = f(x_1)$ and $f_2= h(x_1)+k(x_2)$.
 \item If $\partial_3 f_1 \neq 0$, then $f_2 = 0$.
\end{enumerate}
\end{enumerate}
We have the following result.
\begin{theorem}\label{THeor2Dim3}
 Let $M=\mathbb{R}^{3}$  and $\nabla$ be the torsion free affine connection, whose non-zero coefficients of the 
connection are given by
\begin{equation*}
 \nabla_{\partial_{1}}\partial_{1} = f_{1}\partial_{2},\;\;\nabla_{\partial_{2}}\partial_{2} = f_{2}\partial_{3}\;\;\mbox{and}\;\;\nabla_{\partial_{3}}\partial_{3} = f_{3}\partial_{1}.
\end{equation*} 
Then $(M,\nabla)$ is affine Szab\'o if at least one of the following conditions holds:
\begin{enumerate}
 \item $f_1 =0$, $f_2=u(x_2)$ and $f_3 = v(x_2) + t(x_3)$.
 \item $f_2 =0$, $f_3=t(x_3)$ and $f_1 = f(x_1) + g(x_3)$.
 \item $f_3 =0$, $f_1=f(x_1)$ and $f_2 = h(x_1) + u(x_2)$.
\end{enumerate}
Or at least one of the following conditons holds:
\begin{enumerate}
 \item $f_1 =0$, $f_2 = f(x_1) + g(x_2)$ and $f_3 = 0$.
 \item $f_2 =0$, $f_3=v(x_2) + t(x_3)$ and $f_1 = 0$.
 \item $f_3 =0$, $f_1=f(x_1) + g(x_3)$ and $f_2 = 0$.
\end{enumerate}
\end{theorem}

From Theorem \ref{THeor2Dim3}, one can construct examples of affine Szab\'o connections. As an example, we have the following.

\begin{example}{\rm
The following connections on $\mathbb{R}^3$ whose non-zero Christoffel symbols are given by:
\begin{enumerate}
 \item $\nabla_{\partial_{1}}\partial_{1}=0, \quad \nabla_{\partial_{2}}\partial_{2}=x_2\partial_3,
 \quad \nabla_{\partial_{3}}\partial_{3}=(x_2+x^{2}_{3})\partial_1$;
 \item $\nabla_{\partial_{1}}\partial_{1}=x^{2}_{1}\partial_2, \quad \nabla_{\partial_{2}}\partial_{2}=(x_1+x_2)\partial_3,
 \quad \nabla_{\partial_{3}}\partial_{3}=0$;
\end{enumerate}
are affine Szab\'o.}
\end{example}

\begin{remark}{\rm
The affine connection defined in Example \ref{ExamplImport} has a Ricci tensor which is cyclic parallel but it is not affine Szab\'o. This means that the manifold defined in Example \ref{ExamplImport} is an $L_{3}$-space but not an affine Szab\'o manifold.}
\end{remark}

One has also the following observation.

\begin{theorem}
 Let $(M_1,\nabla_1)$ be an affine Szab\'o at $p_1 \in M_1$ and $(M_2,\nabla_2)$ be an affine Szab\'o at $p_2 \in M_2$.
 Then the product manifold $(M,\nabla):=(M_1\times M_2, \nabla \oplus \nabla_2)$ is affine Szab\'o at $p=(p_1,p_2)$.
\end{theorem}

\begin{proof}
 Let $X=(X_1,X_2)\in T_{(p_1,p_2)} (M_1\times M_2)$ with $X_1\in T_{p_1} M_1$ and $X_2\in T_{p_2} M_2$. Then we have
$
  S^{\nabla}(X) = S^{\nabla_1}(X_1)\oplus S^{\nabla_2}(X_2).
$
So
$
 Spect\{S^{\nabla}(X)\}  = Spect\{S^{\nabla_1}(X_1)\} \cup Spect\{S^{\nabla_2}(X_2)\}  = \{0\}\cup \{0\} = \{0\}.
$
\end{proof}

Affine Szab\'o connections are of interest not only in affine geometry, but also in the study of Pseudo-Riemannian Szab\'o metrics
since they provide some examples without Riemannian analogue by means of the Riemannian extensions.

\section{Riemannian extension construction}

Let $\nabla$ be a torsion free affine connection on an $n$-dimensional affine manifold $M$ and
$T^*M$ be the cotangent bundle of $(M,\nabla)$. In the locally induced coordinates $(u_i,u_{i'})$ on 
$\pi^{-1}(U)\subset T^* M$, the \textit{Riemannian extension} $g_{\nabla}$ is the 
pseudo-Riemannian metric on $T^*M$ of neutral signature $(n,n)$ defined by  
\begin{eqnarray}
 g_{\nabla}= 
\left(
      \begin{array}{cc}
      -2u_{k'}\Gamma^{k}_{ij}&\delta^{j}_{i}\\
      \delta^{j}_{i}&0
         \end{array}
\right),
\end{eqnarray}
with respect to $\partial_1,\cdots,\partial_n,\partial_{1'},\cdots,\partial_{n'} (i,j,k=1,\cdots,n;k'=k+n)$, where $\Gamma^{k}_{ij}$ are the Christoffel symbols of the torsion free affine connection $\nabla$ with respect to $(U,u_i)$.

Let $(M,\nabla)$ be an $3$-dimensional affine manifold. Let $(x_1,x_2,x_3)$ be local coordinates on $M$. We expand
$\nabla_{\partial_i} \partial_j = \sum_k \Gamma_{ij}^{k}\partial_k$ for $i,j,k=1,2,3$ to define the Christoffel symbols  $\Gamma_{ij}^{k}$of
$\nabla$. If $\omega \in T^* M$, we expand $\omega=x_4dx_i + x_5dx_2 + x_6dx_3$ to define the dual fiber coordinates
$(x_4,x_5,x_6)$ thereby obtain canonical local coordinates $(x_1,x_2,x_3,x_4,x_5,x_6)$ on $T^* M$. The Riemannian
extension in the metric of neutral signature $(3,3)$ on the cotangent bundle $T^* M$ given localy by
\begin{eqnarray*}
 g_{\nabla}(\partial_1,\partial_4)&=& g_{\nabla}(\partial_2,\partial_5)=g_{\nabla}(\partial_3,\partial_6)=1,\\
 g_{\nabla}(\partial_1,\partial_1)&=& -2x_4\Gamma_{11}^{1}-2x_5\Gamma_{11}^{2}-2x_6\Gamma_{11}^{3},\\
 g_{\nabla}(\partial_1,\partial_2)&=& -2x_4\Gamma_{12}^{1}-2x_5\Gamma_{12}^{2}-2x_6\Gamma_{12}^{3},\\
 g_{\nabla}(\partial_1,\partial_3)&=& -2x_4\Gamma_{13}^{1}-2x_5\Gamma_{13}^{2}-2x_6\Gamma_{13}^{3},\\
 g_{\nabla}(\partial_2,\partial_2)&=& -2x_4\Gamma_{22}^{1}-2x_5\Gamma_{22}^{2}-2x_6\Gamma_{22}^{3},\\
 g_{\nabla}(\partial_2,\partial_3)&=& -2x_4\Gamma_{23}^{1}-2x_5\Gamma_{23}^{2}-2x_6\Gamma_{23}^{3},\\
 g_{\nabla}(\partial_3,\partial_3)&=& -2x_4\Gamma_{33}^{1}-2x_5\Gamma_{33}^{2}-2x_6\Gamma_{33}^{3}.
 \end{eqnarray*}

\begin{lemma}
Let $(M,\nabla)$ be an $n$-dimensional affine manifold and $(T^* M, g_{\nabla})$ be the cotangent bundle with
the twisted Riemannian extension. Then, we have 
\begin{eqnarray*}
  Spect\{\tilde{\mathcal{S}}(\tilde{X})\} = Spect\{\mathcal{S}^{\nabla}(X)\}
\end{eqnarray*}
\end{lemma}

\begin{proof}
Let $\Gamma_{ij}^{k}$ be the Christoffel symbols of $\nabla$. The non-zero Christoffel symbols $\tilde{\Gamma}_{\alpha \beta}^{\gamma}$ of the Levi-Civita connection of
$g_{\nabla}$ are given by:
\begin{eqnarray*}
 \tilde{\Gamma}^{k}_{ij} &=& \Gamma^{k}_{ij}, \quad \tilde{\Gamma}^{k'}_{i'j} = -\Gamma^{i}_{jk}\quad
 \tilde{\Gamma}^{k'}_{ij'} = -\Gamma^{j}_{ik},\\
 \tilde{\Gamma}^{k'}_{ij} &=& \sum_r x_{r'}\Big(\partial_k \Gamma^{r}_{ij} - \partial_i \Gamma^{r}_{jk}
 -\partial_j \Gamma^{r}_{ik} + 2 \sum_l \Gamma^{r}_{kl}\Gamma^{l}_{ij}\Big);
 \end{eqnarray*}
where $(i,j,k,l,r=1,\cdots,n)$ and $(i'=i+n,j'=j+n,k'=k+n,r'=r+n)$. The non-zero components of the curvature tensor of 
$(T^*M,g_{\nabla})$ up to the usual symmetries are given as follows: we omit $\tilde{R}_{kji}^{h'}$, as it plays no role 
in our considerations.
\begin{eqnarray*}
 \tilde{R}_{kji}^{h}&=& R_{kji}^{h}, \quad \tilde{R}_{kji}^{h'}, \quad \tilde{R}_{kji'}^{h'}=-R_{kjh}^{i},\\
 \tilde{R}_{k'ji}^{h'} &=& R_{hij}^{k};
\end{eqnarray*}
where $R_{kji}^{h}$ are the components of the curavture tensor of $(M,\nabla)$.
Let $\tilde{X}=\alpha_i \partial_i + \alpha_{i'}\partial_{i'}$ and $\tilde{Y}=\beta_i \partial_i + \beta_{i'}\partial_{i'}$
be vectors fields on $T^* M$. Let $X=\alpha_i \partial_i$ and $Y=\beta_i\partial_i$ be the correspoding vectors
fields on $M$. Let $\mathcal{S}^{\nabla}(X)$ be the matrix of the affine Szab\'o operator on $M$ relative
to the basis $\{\partial_i\}$. Then the matrix of the Szab\'o operator $\tilde{\mathcal{S}}(\tilde{X})$ with
respect to the basis $\{\partial_i,\partial_{i'}\}$ have the form
\begin{eqnarray*}
 \tilde{\mathcal{S}}(\tilde{X}) = \left(\begin{array}{cc}
                                         \mathcal{S}^{\nabla}(X)&0\\
                                         *& {}^t\mathcal{S}^{\nabla}(X)
                                        \end{array}
\right).
\end{eqnarray*}
\end{proof}
We have the following results.

\begin{theorem}
 Let $(M,\nabla)$ be a smooth torsion-free affine manifold. Then the following statements are equivalent:
 \begin{enumerate}
  \item[(i)] $(M,\nabla)$ is affine Szab\'o.
  \item[(ii)]  The Riemannian extension $(T^*M,g_{\nabla})$ of $(M,\nabla)$ is a pseudo-Riemannian Szab\'o manifold.
 \end{enumerate}
\end{theorem}
\begin{proof}
 Let $\tilde{X}=\alpha_i \partial_i + \alpha_{i'}\partial_{i'}$ be a vector field on $T^* M$. 
 Then the matrix of the Szab\'o operator $\tilde{S}(\tilde{X})$ with respect to the basis $\{\partial_i,\partial_{i'}\}$ is of the form
 \begin{eqnarray}\label{SzaboMatrix}
 \tilde{\mathcal{S}}(\tilde{X}) = \left(\begin{array}{cc}
                                         \mathcal{S}^{\nabla}(X)&0\\
                                         *& {}^t\mathcal{S}^{\nabla}(X)
                                        \end{array}
\right).
\end{eqnarray}
where $\mathcal{S}^{\nabla}(X)$ is the matrix of the affine Szab\'o operator on $M$ relative
to the basis $\{\partial_i\}$. Note that the characteristic polynomial $P_{\lambda}[\tilde{\mathcal{S}}(\tilde{X})]$ of $\tilde{\mathcal{S}}(\tilde{X})$
and $P_{\lambda}[\mathcal{S}^{\nabla}(X)]$ of $\mathcal{S}^{\nabla}(X)$ are related by 
\begin{eqnarray}\label{CharacteristicPolynomial}
P_{\lambda}[\tilde{\mathcal{S}}(\tilde{X})]=P_{\lambda}[\mathcal{S}^{\nabla}(X)]\cdot P_{\lambda}[{}^t\mathcal{S}^{\nabla}(X)].
\end{eqnarray}
Now, if the affine manifold $(M,\nabla)$ is assumed to be affine Szab\'o, then $\mathcal{S}^{\nabla}(X)$ has zero eigenvalues for each
vector field $X$ on $M$. Therefore, it follows from (\ref{SzaboMatrix}) that the eigenvalues of $\tilde{\mathcal{S}}(\tilde{X})$ vanish
for every vector field $\tilde{X}$ on $T^* M$. Thus $(T^*M, g_{\nabla})$ is pseudo-Riemannian Szab\'o manifold.\\
Conversely, assume that $(T^*M, g_{\nabla})$ is an pseudo-Riemannian Szab\'o manifold. If $X=\alpha_i \partial_i$ with $\alpha_{i}\neq 0$, for any $i$, is a vector field on $M$, then $\tilde{X}=\alpha_i \partial_i + \frac{1}{2\alpha_i}\partial_{i'}$ is an unit vector field at every point of
the zero section on $T^* M$. Then from (\ref{SzaboMatrix}), we see that, the characteristic polynomial 
$P_{\lambda}[\tilde{\mathcal{S}}(\tilde{X})]$ of $\tilde{\mathcal{S}}(\tilde{X})$ is the square of the characteristic polynomial
$P_{\lambda}[\mathcal{S}^{\nabla}(X)]$ of $\mathcal{S}^{\nabla}(X)$. Since for every unit vector field $\tilde{X}$ on $T^* M$ the
characteristic polynomial $P_{\lambda}[\tilde{\mathcal{S}}(\tilde{X})]$ sould be the same, it follows that for every vector field
$X$ on $M$ the characteristic polynomial $P_{\lambda}[\mathcal{S}^{\nabla}(X)]$ is the same. Hence $(M,\nabla)$ is affine Szab\'o.
\end{proof}

As an example, we have the following. The Riemannian extension of the affine Szab\'o connection on $\mathbb{R}^3$
defined by
\begin{eqnarray*}
 \nabla_{\partial_1} \partial_1 =x_1\partial_1, \quad 
 \nabla_{\partial_1} \partial_2 =2x_3\partial_1,\quad
 \nabla_{\partial_1} \partial_3 =-2x_2\partial_1
\end{eqnarray*}
is the pseudo-Riemannian metric of signature $(3,3)$ given by
\begin{align}
 g_{\nabla} &=  2dx_1\otimes dx_4 +2dx_2\otimes dx_5 +2dx_3\otimes dx_6 \nonumber\\
&  -2x_1x_4 dx_1\otimes dx_1 -4x_3x_4dx_1\otimes dx_2 +4x_2x_4dx_1\otimes dx_3.\nonumber
 \end{align}
After, a straightforward calculation, it easy to see that this metric is Szab\'o. 

The Riemannian extensions provide a link between affine and pseudo-Riemannian geometries. Some properties
of the affine connection $\nabla$ can be investigated by means of the corresponding properties of the
Riemannian extension $g_{\nabla}$. For more details and information about Riemannian extensions,
see \cite{broz,calvino2,GarciaGilkeyNikcevicLorenzo,gar} and references therein. For instance, it is known, in \cite{broz, calvino2} and references therein, that a Walker metric is a triple $(M,g,\mathcal{D})$, where $M$ is an $n$-dimensional manifold, $g$ is a pseudo-Riemannian metric on $M$ and $\mathcal{D}$ is an $r$-dimensional parallel null distribution ($r>0$). In \cite{calvino2}, the authors showed that any four-dimensional Riemannian extension is necessarily a self-dual Walker manifold, but for some particular cases, they proved that the converse holds.


\end{document}